\documentclass[a4paper]{article}
\usepackage{amsmath,amssymb}
\usepackage[colorlinks,citecolor=blue]{hyperref}
\usepackage[utf8]{inputenc}
\usepackage[dvips]{epsfig}
\usepackage{tikz}
\usepackage[english]{babel}

\newtheorem{theorem}{Theorem}[section] 
\newtheorem{proposition}[theorem]{Proposition} 
\newtheorem{conjecture}[theorem]{Conjecture} 
 
\newtheorem{lemma}[theorem]{Lemma} 

\newtheorem{remark}[theorem]{Remark}

\newenvironment{proof}{\begin{trivlist}\item{\bf{Proof.}}}
  {\hfill\rule{2mm}{2mm}\end{trivlist}}

\newcommand{\ZZ}{\mathbb{Z}}
\newcommand{\NN}{\mathbb{N}}
\newcommand{\CC}{\mathbb{C}}
\newcommand{\QQ}{\mathbb{Q}}
\newcommand{\RR}{\mathbb{R}}

\newcommand{\qbinom}[2]{\genfrac{[}{]}{0mm}{1}{#1}{#2}}
\newcommand{\qbase}[3]{\genfrac{[}{]}{0mm}{1}{#1,\,#2}{#3}}
\newcommand{\ano}{\mathsf{A}}
\newcommand{\intr}{\operatorname{Int}}
\newcommand{\sh}{\mathsf{s}}  

\def\x{{\bf x}}
\def\a{{\bf a}}
\def\s{{\bf s}}
\newcommand{\alphab}{\boldsymbol{\alpha}}

\title{$q$-Ehrhart polynomials of Gorenstein polytopes,\\Bernoulli umbra and related Dirichlet series}
\author{F. Chapoton\footnote{Cet auteur a bénéficié d'une aide de
    l'Agence Nationale de la Recherche (projet Carma, référence
    ANR-12-BS01-0017).}~ and D. Essouabri}
\date{\today}
\begin{document}

\maketitle

\begin{abstract}
  This article considers some $q$-analogues of classical results
  concerning the Ehrhart polynomials of Gorenstein polytopes, namely
  properties of their $q$-Ehrhart polynomial with respect to a good
  linear form. Another theme is a specific linear form $\Psi$
  (involving Carlitz' $q$-analogues of Bernoulli numbers) on the space
  of polynomials, for which one shows interesting behaviour on these
  $q$-Ehrhart polynomials. A third point is devoted to some related
  zeta-like functions associated with polynomials.
\end{abstract}

\section*{Introduction}

This article deals with properties of the $q$-Ehrhart polynomial of
Gorenstein lattice polytopes, with a specific linear form on
polynomials and with its relationship to some Dirichlet series. But
the initial motivation was something rather different. Let us start
with a short account of this story.

It has proved useful and interesting to consider, as a kind of
non-associative replacement of usual formal power series in one
variable, some infinite formal sums of rooted trees with coefficients
in a base ring, where rooted trees plays the role of the monomials
$x^n$. These objects can then be multiplied and composed, making a
setting very similar to the classical one. One could call them
\textit{tree-indexed series}.

Among all the tree-indexed series, there are two specific ones,
playing a role similar to the usual exponential and logarithm power
series. Let us call them $A$ and $\Omega$
\cite{chapoton_om_qx,chapoton_om}. Then $A$ is an analog of the
exponential and there is a very simple and nice formula for its
coefficients which are all positive rationals. On the contrary,
$\Omega$ is an analog of the logarithm and has very complicated
rational coefficients with signs, some of them vanishing.

One intriguing problem is to understand for which trees does the
coefficient in $\Omega$ vanish. Here the Ehrhart polynomials enters
the game, as it is known that the coefficient $\Omega_T$ of a tree $T$
in $\Omega$ can be expressed using the Ehrhart polynomial of a
polytope attached to $T$.

After looking closely at the trees with vanishing coefficients, one
observed that most of them (but not all) have a very special shape,
namely all their leaves have the same height. It turns out that this
shape implies that the associated polytope is Gorenstein. This was the
starting point of the present article.

Once generalized as much as possible, the problem was then to prove
that a specific linear form $\Psi$ vanishes on the product of Ehrhart
polynomials of two $r$-Gorenstein polytopes of odd total dimension.

It turns out that this vanishing property is best seen and explained
when the classical notion of Ehrhart polynomial is replaced by the
$q$-Ehrhart polynomial introduced in \cite{chapoton_qehrhart}. Instead
of a vanishing property, one has to prove that a specific linear form
takes values that are (up to sign) self-reciprocal elements of
$\QQ(q)$.

The vanishing properties that are obtained as a corollary when setting
$q=1$ are similar to the vanishing of the Bernoulli numbers of odd
indices, which can be seen as a consequence of the functional equation
of the Riemann $\zeta$ function. This inspired by analogy the study of
a Dirichlet series attached to a polynomial, for which one describes
the analytic continuation and obtains a simple expression of values at
negative integers in term of the linear form $\Psi$.

\medskip

Let us now describe the contents of this article.

In section \ref{step1}, one states a simple symmetry property of the
$q$-Ehrhart polynomials of Gorenstein polytopes.

Sections \ref{step2} and \ref{step3} are about the general setting,
sketching the landscape and gathering tools. Section \ref{step2}
introduces a $q$-analog of the space of integer-valued polynomials and
several subspaces and bases. In section \ref{step3}, one introduces a
linear operator $\Sigma$ and two linear forms $V$ and $\Psi$, compute
some of their values and explain the relations of $\Psi$ with some
$q$-analogues of Bernoulli numbers due to Carlitz.

Section \ref{main} contains the main results about $q$-Ehrhart
polynomials. One first obtains a symmetry property for the
coefficients of the $q$-Ehrhart polynomials of Gorenstein polytopes,
when expressed in a particular basis. This statement is a $q$-analog
of the well-known symmetry of the coefficients of the numerator of the
Ehrhart series of Gorenstein polytopes. Using this symmetry, one then
proves that the image by $\Psi$ of the product of two Ehrhart
polynomials of $r$-Gorenstein polytopes is (up to sign)
self-reciprocal. A similar result is obtained in the special case of
$1$-Gorenstein (a.k.a. reflexive) polytopes.

In section \ref{classic}, one lets $q=1$ and goes back to the classical
setting of Ehrhart polynomials. One first states the desired vanishing
results, as easy consequences of the main results. This is illustrated
by several examples. One also proposes a conjecture about what happens
when one considers the powers of one fixed Gorenstein polytope. The
sequence of rational numbers thus obtained seems to share properties
with the Bernoulli numbers. This suggest to see them as the values at
negative integers of a zeta-like function, by analogy with the
classical relation between Bernoulli numbers and Riemann zeta
function. This leads to the definition of a Dirichlet
series attached to the Ehrhart polynomial.

In the last section \ref{meromorph}, one proves that these Dirichlet
series have a meromorphic continuation to $\CC$ with just a simple
pole at $1$, and that their values at negative integers are given by
the expected formula.

\medskip

Many thanks to Éric Delaygue, Frédéric Jouhet, Philippe Nadeau and
Tanguy Rivoal for useful discussions and their help to navigate in the
hypergeometric ocean.

\subsection{Notations}

Let us introduce some basic notations. The letter $q$ will always
stand for a formal parameter. It will either be considered as an
element of $\QQ(q)$ or as an element of $\ZZ[q,1/q]$.

For $n\in\NN$, we will denote by $[n]_q$ the $q$-integer $1+q+\dots +
q^{n-1}$. Using instead the formula $\frac{q^n-1}{q-1}$, this can be
extended to $n \in \ZZ$. One then has the obvious relations $[n]_{1/q}
= q^{-n+1} [n]_q$ and $[-n]_q = - q^{-n} [n]_q$.

For $n \in \NN$, we will denote by $[n]!_q$ the $q$-factorial of $n$,
namely the product $[1]_q[2]_q \dots [n]_q$.

For $m,n \in \NN$, we will denote by $\qbinom{m}{n}_q$ the
$q$-binomial $\frac{[m]!_q}{[n]!_q[m-n]!_q}$.

For fixed $n\in \NN$, this can be written as $([m]_q [m-1]_q \dots
[m-n+1]_q)/[n]!_q$, which makes sense for every $m \in \ZZ$. One then
has the useful formulas $\qbinom{-m}{n}_q = (-1)^n q^{-n m
  -\binom{n}{2}}\qbinom{m+n-1}{n}_q$ and $\qbinom{m}{n}_{1/q}=
q^{-m(n-m)} \qbinom{-m}{n}_q$.

\medskip

All the previous notations are rather standard for these
$q$-analogues. One will need also some other notations, less classical.

For $n \in \NN$, let $[n,x]_q$ be the polynomial $[n]_q + q^{n}x$.
For $m,n \in \NN$, let also define the polynomial
\begin{equation*}
\qbase{m}{x}{n}_q = \frac{[m-n+1,x]_q[m-n+2,x]_q\dots [m,x]_q}{[n]_q!}.
\end{equation*}
When $q$ is replaced by $1$, they become $n+x$ and $\binom{m+x}{n}$.

These polynomials are defined in this way so that they have nice
evaluations when $x$ is replaced by a $q$-integer $[k]_q$. Indeed
$[n,[k_q]]_q$ is just $[n+k]_q$ and therefore
$\qbase{m}{[k]_q}{n}_q=\qbinom{m+k}{n}_q$.

They also satisfy the following translation properties:
\begin{equation}
[n,[k,x]_q]_q = [n+k, x]_q\quad\text{and}\quad  \qbase{m}{[k,x]_q}{n}_q = \qbase{m+k}{x}{n}_q.
\end{equation}

\section{On $q$-Ehrhart polynomials of Gorenstein polytopes}

\label{step1}

One will use here some results of the article \cite{chapoton_qehrhart},
where a $q$-analogue of the classical theory of Ehrhart polynomial has
been introduced.

Recall that a lattice polytope $P$ is called \textit{reflexive} if it
contains the lattice origin $0$ and the dual polytope $P^*$ is also a
lattice polytope. These polytopes are used in the study of mirror
symmetry in the setting of toric geometry. There is another closely
related notion. A lattice polytope $P$ is called
$r$-\textit{Gorenstein} (for some integer $r \geq 1$) if the dilated
polytope $r P$ is (up to lattice translation) reflexive.

For more on reflexive and Gorenstein polytopes, the reader can consult 
for example \cite{batyrev_jag,batyrev_nill,nill,nill_gor,nill_schepers}.

Let us now assume that $P$ is an $r$-Gorenstein lattice polytope of
dimension $D$. Let $z_0$ be the unique interior lattice point of $r
P$.

Let $\lambda$ be a linear form on the lattice, such that $\lambda$ is
positive on $P$ and $\lambda$ is not constant on any edge of
$P$. These conditions are required for the definition of the
$q$-Ehrhart polynomial (see \cite{chapoton_qehrhart} for details).

Then one can consider the $q$-Ehrhart polynomial $E_{P,\lambda}(x,q)$, defined by
\begin{equation}
  E_{P,\lambda}([n]_q, q) = \sum_{s \in n P} q^{\lambda(s)}.
\end{equation}
For short, it will be denoted by $E$ when no ambiguity is
possible. This is a polynomial in $\QQ(q)[x]$.

Our first result is the following simple symmetry property.

\begin{proposition}
  \label{symmetry_of_qehr}
  The $q$-Ehrhart polynomial $E$ satisfies
  \begin{equation}
    E(x, q) = (-1)^D E(-q [r,x]_q, 1/q) q^{-\lambda(z_0)}.
  \end{equation}
\end{proposition}

This is a $q$-deformation of the classical relation
\begin{equation}
  E(x) = (-1)^D E(-r-x)
\end{equation}
for the Ehrhart polynomial of $r$-Gorenstein lattice polytopes.

\begin{proof}
  Let $n\geq 1$ be an integer.  By $q$-Ehrhart reciprocity
  \cite[Th. 2.5]{chapoton_qehrhart}, one knows that
  \begin{equation*}
    E([-n]_q,q) = (-1)^D \sum_{s \in \intr (n P)} q^{-\lambda(s)}.
  \end{equation*}
  By the Gorenstein property, the translation by the vector $z_0$
  gives an isomorphism of lattice polytopes from $n P$ to $\intr ((n+r)
  P)$. It follows that
  \begin{equation*}
    E([-n-r]_q, q) = (-1)^D \sum_{s \in n P} q^{-\lambda(s+z_0)},
  \end{equation*}
  whose right hand side can be written as
  \begin{equation*}
    (-1)^D E([n]_q, q) \bigg\vert_{q=1/q} q^{-\lambda(z_0)}=(-1)^D E([n]_{1/q}, 1/q) q^{-\lambda(z_0)} .
  \end{equation*}
  Then consider the variables 
  \begin{equation*}
    x = [-n-r]_{q}, \quad
    X = [n]_{1/q}.
  \end{equation*}
  One can check that they are related by $X = -q [r,x]_q$. The
  statement follows.
\end{proof}

Note that the product $P \times Q$ of two $r$-Gorenstein polytopes is
still an $r$-Gorenstein polytope. Moreover, the $q$-Ehrhart polynomial
of a product $P \times Q$ of polytopes, with respect to the linear
form $\lambda \oplus \mu$, is the product $E_{P,\lambda}
E_{Q,\mu}$. This is obviously compatible with the proposition.

\medskip

If $P$ is a polytope, let us call the \textbf{pyramid} over $P$ the
convex hull of $(0,0)$ and $1 \times P$ in a lattice of one more
dimension. The pyramid over an $r$-Gorenstein polytope is an
$r+1$-Gorenstein polytope.

\section{Binomial bases for polynomials in $x$}

\label{step2}

Let us now consider the polynomial ring $\QQ(q)[x]$ and some of its
elements.

This ring has a basis over $\QQ(q)$ given by the polynomials
$\qbase{n}{x}{n}_q$ for $n \geq 0$, which will be called the
$B$-basis.

Let us now define $\ano_q$, as the subspace of $\QQ(q)[x]$ generated
over $\ZZ[q,1/q]$ by the polynomials $\qbase{n}{x}{n}_q$.

For an integer $d \in \NN$, let us denote by $\ano^{(d)}_q$ the
subspace of $\ano_q$ of polynomials of degree at most $d$.

\begin{proposition}
  \label{h_base}
  The polynomials $\qbase{k}{x}{d}_q$ for $k=0,\dots,d$ form a basis
  of $\ano^{(d)}_q$ over $\ZZ[q,1/q]$.
\end{proposition}
\begin{proof}
  This follows from lemma \ref{square_matrix}, which proves that the
  matrix of coefficients of these polynomials in the $B$-basis is
  triangular with powers of $q$ on the diagonal.
\end{proof}

\begin{lemma}
  \label{square_matrix}
  For integers $0 \leq i \leq d$, there holds
  \begin{equation}
    \qbase{i}{x}{d}_q = \sum_{j=0}^{d} (-1)^{d-j} q^{-d(d-i)+\binom{d-j}{2}} \qbinom{d-i}{d-j}_q \qbase{j}{x}{j}_q.
  \end{equation}
\end{lemma}
\begin{proof}
  It is enough to check this for all $q$-integers $[k]_q$. This becomes
  \begin{equation*}
    \qbinom{i+k}{d}_q = \sum_{j=0}^{d} (-1)^{d-j} q^{-d(d-i)+\binom{d-j}{2}}\qbinom{d-i}{d-j}_q \qbinom{j+k}{j}_q.
  \end{equation*}
  This is an instance of the $q$-Chu-Vandermonde formula for the
  ${}_2\phi_{1}$ basic hypergeometric function, see for example
  \cite[Appendix II, formula (II.7)]{gasper}.
\end{proof}

Let us now describe the product in this basis.

\begin{proposition}
  \label{gamma}
  For all integers $0 \leq i \leq d$ and $0 \leq j \leq e$, there holds
  \begin{equation}
    \qbase{i}{x}{d}_q \qbase{j}{x}{e}_q = \sum_{0 \leq \ell \leq d+e} q^{(\ell-e-i)(\ell-d-j)} \qbinom{d+j-i}{\ell-i}_q\qbinom{e+i-j}{\ell-j}_q\qbase{\ell}{x}{d+e}_q.
  \end{equation}
\end{proposition}
\begin{proof}
  As this is an equality of polynomials in $x$, it is enough to check
  that it holds for all positive $q$-integers, namely that
  \begin{equation*}
    \qbinom{i+k}{d}_q \qbinom{j+k}{e}_q = \sum_{0 \leq \ell \leq d+e} q^{(\ell-e-i)(\ell-d-j)} \qbinom{d+j-i}{\ell-i}_q\qbinom{e+i-j}{\ell-j}_q\qbinom{\ell+k}{d+e}_q
  \end{equation*}
  holds for all $k \geq 0$. This equality is in fact an instance of
  the classical Pfaff-Saalsch{\"u}tz identity for the basic
  hypergeometric function ${}_{3}\phi_{2}$. It can be recovered for
  example by letting $d=-j, a=e+i, e=-i, b=d+j, c=-k-1$ in formula
  $(4)$ of \cite{zeng_pfaff}.
\end{proof}

Proposition \ref{h_base} and \ref{gamma} together implies that the subspace $\ano_q$ of $\QQ(q)[x]$ is a
commutative ring over $\ZZ[q,1/q]$.

Let us now turn to a simple symmetry statement, for later use.
\begin{proposition}
  \label{symmetry_of_H_basis}
  For all integers $d,r,k$, the polynomials $\qbase{k}{x}{d}_q$ have
  the following symmetry property:
  \begin{equation}
    \qbase{k}{-q[r,x]_q}{d}_{1/q} = (-1)^d q^{\binom{d+1}{2}} \qbase{r-1+d-k}{x}{d}_q.
  \end{equation}
\end{proposition}

\begin{proof}
  This is a simple computation using the definition of these
  polynomials. The left hand side is
  \begin{equation*}
    \frac{[k-d+1,y]_{1/q} \dots [k,y]_{1/q}}{[1]_{1/q}\dots [d]_{1/q}}
  \end{equation*}
  with $y = -q [r,x]_q$. This can be rewritten as
  \begin{equation*}
    \frac{q^{d-k}([k-d+1]_q-[r,x]_q)\dots q^{1-k}([k]_q-[r,x]_q)}{q^{0}[1]_q \dots q^{-d+1}[d]_q}.
  \end{equation*}
  This becomes
  \begin{equation*}
    (-1)^d q^{\binom{d}{2}} \frac{q(q^{r-1+d-k}x+[r-1+d-k]_q) \dots q(q^{r-k}x+[r-k]_q)}{[d]!_q},
  \end{equation*}
  which gives the expected result.
\end{proof}

\section{Operator and linear forms}
\label{step3}

Let us define an endomorphism $\Sigma$ of $\QQ(q)[x]$ by
\begin{equation}
  (\Sigma E)([n]_q) = \sum_{j=0}^{n} q^{j} E([j]_q),
\end{equation}
for all polynomials $E$.

If $E$ is the $q$-Ehrhart polynomial $E_{P,\lambda}$ of a polytope $P$
and linear form $\lambda$, then $\Sigma E$ is the $q$-Ehrhart
polynomial of the pyramid over $P$ as defined at the end of section \ref{step1},
with the linear form $1\oplus \lambda$.

\begin{lemma}
  For every integer $d \geq 0$, there holds
  \begin{equation}
    \label{sigma_on_basis}
    \Sigma \qbase{d}{x}{d}_q = \qbase{d+1}{x}{d+1}_q.
  \end{equation}
\end{lemma}
\begin{proof}
  As an equality between polynomials in $x$, it is enough to check
  that it holds for every positive $q$-integer $k$. This becomes
  \begin{equation*}
    \label{sum_qbinom_1}
    \sum_{j=0}^{k} q^j \qbinom{d+j}{d}_q = \qbinom{d+1+k}{d+1}_q.
  \end{equation*}
  This is a classical formula, which has a simple combinatorial proof
  using the description of $q$-binomials by paths in a rectangle
  according to their area.
\end{proof}
Note that property \eqref{sigma_on_basis} uniquely defines the linear
operator $\Sigma$. This also proves that it acts on the subring $\ano_q$.

\begin{lemma}
  For all integers $i$ and $d$ with $0 \leq i \leq d$, there holds
  \begin{equation}
    \label{sigma_qbase}
    \Sigma \qbase{i}{x}{d}_q = q^{d-i} (\qbase{i+1}{x}{d+1}_q - \qbinom{i}{d+1}_q).
  \end{equation}
\end{lemma}
\begin{proof}  
  To prove this equality of polynomials in $x$, it is enough to check
  the statement for every positive $q$-integer $[k]_q$. This becomes
  \begin{equation*}
    \sum_{j=0}^{k} q^j \qbinom{i+j}{d}_q = q^{d-i} (\qbinom{i+1+k}{d+1}_q - \qbinom{i}{d+1}_q).
  \end{equation*}
  This holds because
  \begin{equation*}
    \sum_{j=0}^{k} q^{j-d} \qbinom{j}{d}_q = \qbinom{k+1}{d+1}_q,
  \end{equation*}
  which is a classical formula, equivalent to \eqref{sum_qbinom_1}.
\end{proof}

Let us define next a linear form $V$ from $\ano_q$ to $\QQ(q)$ by
\begin{equation}
  V(E) = \lim_{x \to [-1]_q} \frac{E(x)-E([-1]_q)}{1 + q x}.
\end{equation}
Note that $[-1]_q = -1/q$. This operator is therefore essentially the
derivative of $E$ at $x = [-1]_q$, up to a multiplicative factor of $q$.

Let us now define another linear form $\Psi$ on $\ano_q$ by the
composition
\begin{equation}
  \Psi(E) = V \Sigma E.
\end{equation}
One will later study the values of the linear form $\Psi$ on the
$q$-Ehrhart polynomials of Gorenstein polytopes.

Let us first compute the values of $\Psi$ on the basis elements.
\begin{proposition}
  \label{beta}
  For all integers $0 \leq i \leq d$, there holds
  \begin{equation}
    \Psi(\qbase{i}{x}{d}_q) = \frac{(-1)^{d-i}q^{-\binom{d-i}{2}}}{[d+1]_q \qbinom{d}{i}_q}.
  \end{equation}
\end{proposition}
\begin{proof}
  Using formula \eqref{sigma_qbase}, one can compute
  \begin{equation*}
    \Psi \qbase{i}{x}{d}_q 
    = q^{d-i} V \left( \qbase{i+1}{x}{d+1}_q - \qbinom{i}{d+1}_q  \right)
    = q^{d-i} V \left( \frac{[i-d+1,x]_q \dots [i+1,x]_q }{[d+1]!_q}-\frac{ [i-d]_q\dots [i]_q}{[d+1]!_q}\right).
  \end{equation*}
  By definition, the operator $V$ is proportional to the derivative at
  $[-1]_q$. This implies that one gets
  \begin{equation*}
    q^{d-i} \frac{([i-d]_q \dots [-1]_q) ([1]_q\dots [i]_q )}{[d+1]!_q},
  \end{equation*}
  which can be readily rewritten as the expected result.
\end{proof}

From these values, one deduces the following lemma.
\begin{lemma}
  \label{psi_and_shift}
  For every polynomial $E \in \QQ(q)[x]$, there holds
  \begin{equation}
    q \Psi(E(1+q x)) - \Psi(E) = (q-1) E(0) + \partial_x E(0).
  \end{equation}
\end{lemma}
\begin{proof}
  As both sides are linear in $E$, it is enough to check this identity
  for every basis element $E = \qbase{d}{x}{d}_q$. First note that
  $\Psi(E)=\frac{1}{[d+1]_q}$ by proposition \ref{beta}. Then using
  \eqref{sigma_qbase}, one computes
  \begin{equation*}
    q \Psi(E(1+q x)) = q V \Sigma (E(1+q x)) = V (\qbase{d+2}{x}{d+1} - \qbinom{d+1}{d+1}) =  V (\qbase{d+2}{x}{d+1}).
  \end{equation*}
  By a direct computation using that $V$ is proportional to the
  derivative at $[-1]_q$, this is $ \sum_{j=1}^{d+1}
  \frac{q^j}{[j]_q}$.  The same computation gives that $
    \partial_x E(0) = \sum_{j=1}^{d} \frac{q^j}{[j]_q}$.
  The result follows.
\end{proof}

Let us now introduce the $q$-Bernoulli numbers of Carlitz by the formula
\begin{equation}
  \label{defiB}
  \Psi(x^n) = B_{q,n},
\end{equation}
for $n \geq 0$.

These rationals fractions, introduced by Carlitz in
\cite{carlitz_beta}, are $q$-analogues of the Bernoulli numbers with
nice properties. In particular, they only have simples poles at non-trivial
roots of unity, and their value at $q=1$ are the classical Bernoulli
numbers. To see that \eqref{defiB} gives the same definition as
Carlitz one, one can use lemma \ref{psi_and_shift} applied to the
monomials $x^n$.

Let us now go back to the study of $\Psi$. One will need the following
result later.
\begin{proposition}
  \label{alpha}
  For all integers $0 \leq i \leq d$ and $0 \leq j \leq e$, there holds
  \begin{equation}
    \Psi(\qbase{i}{x}{d}_q\qbase{j}{x}{e}_q) = \frac{(-1)^{d-i+e-j}q^{-\binom{d-i}{2}+(d-i)(e-j)-\binom{e-j}{2}}}{[d+e+1]_q \qbinom{d+e}{d-i+j}_q}.
  \end{equation}
\end{proposition}
\begin{proof}
  Let us compute $\Psi(\qbase{i}{x}{d}_q\qbase{j}{x}{e}_q)$. By
  proposition \ref{gamma}, this is
  \begin{equation*}
    \sum_{0 \leq \ell \leq d+e} q^{(\ell-e-i)(\ell-d-j)} \qbinom{d+j-i}{\ell-i}_q\qbinom{e+i-j}{\ell-j}_q\Psi(\qbase{\ell}{x}{d+e}_q).
  \end{equation*}

  By proposition \ref{beta}, this is 
  \begin{equation*}
    \sum_{0 \leq \ell \leq d+e} q^{(\ell-e-i)(\ell-d-j)} \qbinom{d+j-i}{\ell-i}_q\qbinom{e+i-j}{\ell-j}_q\frac{(-1)^{d+e-\ell}q^{-\binom{d+e-\ell}{2}}}{[d+e+1]_q \qbinom{d+e}{\ell}}.
  \end{equation*}

  Using lemma \ref{nut}, this becomes the expected result.
\end{proof}

\begin{lemma}
  \label{nut}
  Let $0 \leq i \leq d$ and $0 \leq j \leq e$ be integers. Then
  \begin{equation*}
    \sum_{0 \leq \ell \leq d+e} (-1)^\ell q^{(\ell-e-i)(\ell-d-j)-\binom{d+e-\ell}{2}}\frac{\binom{e}{\ell-i}_q\binom{d}{\ell-j}_q}{\binom{d+e}{\ell}_q} = \frac{(-1)^{i+j} q^{-\binom{d-i}{2}+(d-i)(e-j)-\binom{e-j}{2}}}{ \binom{d+e}{d}_q}.
  \end{equation*}
\end{lemma}
\begin{proof}
  One can assume without loss of generality that $i \geq j$. This can
  then be reformulated as an hypergeometric identity for the function
  ${}_3\phi_{2} $. This formula can be deduced from \cite[Appendix III, formula (III.10)]{gasper}.
\end{proof}

\section{Symmetry of coefficients and self-reciprocal values}

\label{main}

Let $P$ be an $r$-Gorenstein lattice polytope of dimension $D$. Let
$E(x,q)$ be its $q$-Ehrhart polynomial with respect to a linear form
$\lambda$.

Let $d$ be the degree of $E(x,q)$. Using proposition \ref{h_base}, let
us write $E(x,q)$ as follows:
\begin{equation}
  \label{coeff_E}
  E(x,q) = \sum_{j=0}^{d} c_j \qbase{j}{x}{d},
\end{equation}
for some coefficients $c_j$ in $\QQ(q)$.

\begin{proposition}
  \label{sym_and_vanish}
  The coefficients $c_k$ vanish for $0 \leq k \leq r-2$. Moreover
  \begin{equation}
    \label{sym_coeff_gorenstein}
    c_k = (-1)^{D+d} q^{\binom{d+1}{2}-\lambda(z_0)} c_{r-1+d-k}(1/q).
  \end{equation}
\end{proposition}

\begin{proof}
  Because $P$ is an $r$-Gorenstein polytope, the dilated polytopes $k
  P$ have an empty interior if $1 \leq k \leq r-1$. This implies that
  $E(x,q)$ vanishes at the $q$-integers $[-1]_q,\dots,[1-r]_q$. This
  in turn implies the vanishing of the coefficients $c_0, \dots,
  c_{r-2}$ (by an easy induction).

  Let us now show that the symmetry property of proposition
  \ref{symmetry_of_H_basis} together with the symmetry property of
  proposition \ref{symmetry_of_qehr} implies the expected symmetry of
  the coefficients. One computes
  \begin{multline}
    (-1)^D E(-q[r,x]_q,1/q) q^{-\lambda(z_0)} 
=   (-1)^D \sum_{k=r-1}^{d} c_k(1/q) \qbase{k}{-q[r,x]_q}{d}_{1/q} q^{-\lambda(z_0)}\\
 =  (-1)^D \sum_{k=r-1}^{d} c_k(1/q) (-1)^d q^{\binom{d+1}{2}}\qbase{r-1+d-k}{x}{d}_{q} q^{-\lambda(z_0)}   \\
 =  (-1)^D \sum_{k=r-1}^{d} c_{r-1+d-k}(1/q) (-1)^d q^{\binom{d+1}{2}}\qbase{k}{x}{d}_{q} q^{-\lambda(z_0)}.
  \end{multline}
  One then identifies the coefficients with \eqref{coeff_E} to get the expected equality.
\end{proof}

This statement is a $q$-analogue of the usual symmetry $c_k =
(-1)^{D+d} c_{r-1+d-k}$ for $r$-Gorenstein polytopes. In the classical
setting, the numbers $c_k$ are the coefficients of the numerator of
the Ehrhart series ($h$-vector).

\bigskip

Let now $r \geq 1$ be a fixed integer. Let $P$ and $Q$ be two
$r$-Gorenstein lattice polytopes of dimensions $D$ and $E$. Let $E_P$
and $E_Q$ be their $q$-Ehrhart polynomials with respect to some linear
forms $\lambda$ and $\mu$ (omitted to keep the notation short). Let
$d$ and $e$ be the degrees of these polynomials.

Let us introduce the shortcuts $Z=\lambda(z_0)$ and $Z'=\mu(z'_0)$
where $z_0$ and $z'_0$ are the unique interior points in the dilated
polytopes $r P$ and $r Q$. Let us write
\begin{equation}
  \label{EP_EQ_dans_base_H}
  E_p = \sum_{0 \leq i\leq d} c_i \qbase{i}{x}{d}_q \quad\text{and}\quad E_Q= \sum_{0\leq j\leq e} c'_j  \qbase{j}{x}{e}_q.
\end{equation}

Let $\sh_{-k}$ be the shift (with offset $-k$) defined by $\sh_{-k}(P)([n]_q)
= P([n-k]_q)$ for all $n \in \ZZ$ or equivalently by $(\sh_{-k} P)(x) = P([-k,x])$. Then one has
\begin{equation}
  \sh_{-k} E_p = \sum_{0 \leq i\leq d} c_{i+k} \qbase{i}{x}{d}_q \quad\text{and}\quad \sh_{-k} E_Q= \sum_{0\leq j\leq e} c'_{j+k}  \qbase{j}{x}{e}_q,
\end{equation}
for all $0 \leq k \leq r-1$ (using the vanishing statement in
proposition \ref{sym_and_vanish}).

\begin{theorem}
  \label{th_gorenstein}
  For all $0 \leq k \leq r-1$, the value at $1/q$ of the fraction
  $q^{-k} \Psi(\sh_{-k} (E_P E_Q))$ is $(-1)^{D+E} q^{Z+Z'+r-1}$ times itself.
\end{theorem}

\begin{proof}
  Let us first compute $\Psi(\sh_{-k} (E_P E_Q))$ using the expressions
  \eqref{EP_EQ_dans_base_H} and proposition \ref{alpha}. One gets
  \begin{equation*}
     \frac{(-1)^{d+e}}{[d+e+1]_q}\sum_{i,j} (-1)^{i+j}c_{i+k} c'_{j+k} \frac{q^{-\binom{d-i}{2}+(d-i)(e-j)-\binom{e-j}{2}} }{\qbinom{d+e}{d-i+j}_q}.
  \end{equation*}

  Let us now replace $q$ by $1/q$ in this expression. One gets
  \begin{equation*}
    \frac{(-1)^{d+e}}{[d+e+1]_{1/q}}\sum_{i,j} (-1)^{i+j}c_{i+k}(1/q) c'_{j+k}(1/q) \frac{q^{\binom{d-i}{2}-(d-i)(e-j)+\binom{e-j}{2}}}{\qbinom{d+e}{d-i+j}_{1/q}}.
  \end{equation*}
  Using \eqref{sym_coeff_gorenstein} once for $P$ and once for $Q$, this becomes
  \begin{equation*}
    \frac{(-1)^{D+E}q^{d+e}}{[d+e+1]_{q}}\sum_{i,j} (-1)^{i+j}c_{r-1+d-i-k} c'_{r-1+e-j-k} \frac{q^{\binom{d-i}{2}-(d-i)(e-j)+\binom{e-j}{2}-\binom{d+1}{2}+Z-\binom{e+1}{2}+Z'+(d-i+j)(e-j+i)} }{\qbinom{d+e}{d-i+j}_{q}}.
  \end{equation*}
  Changing the indices of summations $i \longleftrightarrow r-1+d-i-2k$ and $j \longleftrightarrow r-1+e-j-2k$, one gets (after simplifications in the powers of $q$)
  \begin{equation*}
    \frac{q^{-2k} q^{Z+Z'+r-1} (-1)^{D+E}}{[d+e+1]_{q}}\sum_{i,j} (-1)^{d-i+e-j}c_{i+k} c'_{j+k} \frac{q^{-\binom{d-i}{2}+(d-i)(e-j)-\binom{e-j}{2}} }{\qbinom{d+e}{e+i-j}_{q}}.    
  \end{equation*}
  Up to the power of $q$ in front of the sum, this is $(-1)^{D+E}$
  times the initial expression for $\Psi(\sh_{-k} (E_P E_Q))$.
\end{proof}

In the special case of reflexive (\textit{i.e.} $1$-Gorenstein)
polytopes, it is not necessary to consider a product of two polytopes
to obtain a similar result. Let us now assume that $P$ is reflexive.
\begin{theorem}
  \label{th_reflexif}
  The value at $1/q$ of the fraction
  $\Psi(E_P)$ is $(-1)^{D} q^{Z}$ times itself.
\end{theorem}
\begin{proof}
  Let us first compute $\Psi(E_P)$ using the expression
  \eqref{EP_EQ_dans_base_H} for $P$ and proposition \ref{beta}. One gets
  \begin{equation*}
     \frac{(-1)^{d}}{[d+1]_q}\sum_{i} (-1)^{i}c_i \frac{q^{-\binom{d-i}{2}} }{\qbinom{d}{i}_q}.
  \end{equation*}  
  Let us now replace $q$ by $1/q$ in this expression. One gets
  \begin{equation*}
    \frac{(-1)^{d}}{[d+1]_{1/q}}\sum_{i} (-1)^{i}c_i(1/q) \frac{q^{\binom{d-i}{2}}}{\qbinom{d}{i}_{1/q}}.
  \end{equation*}
  Using \eqref{sym_coeff_gorenstein} for $P$ (and the hypothesis
  $r=1$), this becomes
  \begin{equation*}
    \frac{(-1)^{D}q^{d}}{[d+1]_{q}}\sum_{i} (-1)^{i}c_{d-i} \frac{q^{\binom{d-i}{2}-\binom{d+1}{2}+Z + i(d-i)} }{\qbinom{d}{i}_{q}}.
  \end{equation*} 
 Changing the index of summation $i \longleftrightarrow d-i$, one gets (after simplifications in the powers of $q$)
  \begin{equation*}
    \frac{q^{Z} (-1)^{D}}{[d+1]_{q}}\sum_{i} (-1)^{d-i}c_{i} \frac{q^{-\binom{d-i}{2}} }{\qbinom{d}{i}_{q}}.    
  \end{equation*}
  Up to the power $q^Z$ in front of the sum, this is $(-1)^{D}$
  times the initial expression for $\Psi(E_P)$.
\end{proof}

In fact, the self-reciprocal fractions involved in
theorem \ref{th_gorenstein} are all the same. Keeping the same
notations, one has the following result.

\begin{proposition}
  The fractions $q^{-k} \Psi(\sh_{-k} (E_P E_Q))$ for
  $k=0,1,\dots,r-1$ are all equal.
\end{proposition}
\begin{proof}
  Let us apply lemma \ref{psi_and_shift} to the polynomial $\sh_{-k}
  (E_P E_Q)$ for $k = 1, \dots, r-1$. One gets
  \begin{multline*}
    q \Psi(\sh_{1-k} (E_P E_Q)) - \Psi(\sh_{-k} (E_P E_Q)) \\= (q - 1)
    E_P([-k]_q)E_Q([-k]_q) + q^{-k} \partial_x E_P([-k]_q) E_Q([-k]_q) +
    q^{-k} E_P([-k]_q) \partial_x E_Q([-k]_q).
  \end{multline*}
  Because of the $r$-Gorenstein property, the right hand side
  vanishes for $k=1, \dots, r-1$. This implies the statement.
\end{proof}
\section{Classical case $q=1$}

\label{classic}

One can state purely classical corollaries of theorems
\ref{th_gorenstein} and \ref{th_reflexif} by letting $q=1$. The
$q$-Ehrhart polynomial becomes the Ehrhart polynomial, and does no
longer depend on the choice of a linear form $\lambda$.

In this context, $\Psi$ becomes the linear form on the space $\QQ[x]$
that maps $x^n$ to the Bernoulli number $B_n$. The operator $s_{-k}$
becomes the evaluation of polynomials in $x$ at $x-k$. One obtains the
following statements.

\begin{theorem}
  \label{zero_odd_3}
  Let $P$ be a product of at least two $r$-Gorenstein polytopes. Let
  $E_P$ be its Ehrhart polynomial. The numbers $\Psi(\sh_{-k} E_P)$
  for $k=0,1,\dots,r-1$ are all equal. If moreover the dimension of
  $P$ is odd, then they all vanish.
\end{theorem}

\begin{theorem}
  \label{zero_refl}
  Let $P$ be a reflexive polytope. Let $E_P$ be its Ehrhart
  polynomial. If the dimension of $P$ is odd, then $\Psi(E_P)=0$.
\end{theorem}

Let us now consider some simple examples.

Let $P$ be the polytope with vertices $0$ and $1$ in $\ZZ$. This is a
$2$-Gorenstein polytope, with Ehrhart polynomial $x+1$. One deduces
from theorem \ref{zero_odd_3} that $\Psi((x+1)^n) = 0$ for all odd $n
\geq 3$. By definition of the Bernoulli numbers, the expression
$\Psi((x+1)^n)$ is just $B_n$ itself, which is well-known to vanish in
this case.

Let now $P$ be the polytope with vertices $0$ and $2$ in $\ZZ$. This
is a reflexive polytope (up to translation), with Ehrhart polynomial $1 + 2 x$. Therefore theorem \ref{zero_refl} implies that $\Psi(1+2x)=0$, which is indeed
the case because $B_0=1$ and $B_{1}=-1/2$.

Let us consider a more complicated example. There exists a reflexive
simplex in dimension $5$ with $355785$ lattice points \cite{nill}. Its
Ehrhart polynomial is $E = 271803 x^5/5 + 271803 x^4/2 + 118594 x^3 +
83979 x^2/2 + 24692 x/5 + 1$. One can check directly that its image by
$\Psi$ vanishes, as well as the images by $\Psi$ of its small odd
powers. By contrast, the even values do not vanish, for example
$\Psi(E^2) = -48827203879/165$.

As an interesting counter-example, consider the triangle in $\ZZ^2$
with vertices $(0,0)$, $(1,0)$ and $(1,1)$. The Ehrhart polynomial is
$E=\binom{x+2}{2}$. The first few values of $\Psi(E^i)$ are given by
\begin{equation}
  \label{rama}
  1, 1/3, 1/30, -1/105, 1/210, -1/231, 191/30030, -29/2145, 2833/72930, \dots
\end{equation}
There is no vanishing here, as this polytope is $3$-Gorenstein, but
not a product of two such polytopes. One can note that these
coefficients have appeared in the work of Ramanujan, in an asymptotic
formula involving triangular numbers (see number (9) of \cite[Chapter
38]{berndtV}).

\subsection{Bernoulli-like numbers attached to Gorenstein polytopes}

Let $P$ be an $r$-Gorenstein polytope of odd dimension $D$ and let
$E_P$ be its Ehrhart polynomial. As a special case of theorem
\ref{zero_odd_3}, the rationals numbers $\Psi(E_P^k)$ attached to the
powers $E_P^k$ vanish for every odd integer $k \geq 3$. This can be
seen as an analog of the same statement for Bernoulli
numbers.

This suggest, for any fixed $r$-Gorenstein polytope $P$, to think
about the sequence $\Psi(E_p^k)_{k \geq 0}$ as some kind of
Bernoulli-like numbers attached to the Gorenstein polytope $P$.

It seems that at least one other property of Bernoulli numbers
extend to the Bernoulli-like numbers, namely the following alternating
sign property, which is is well-known for the Bernoulli numbers.
\begin{conjecture}
  If the dimension of $P$ is odd, the signs of the non-zero
  $\Psi(E_{P}^k)$ alternate.
\end{conjecture}

For example, consider the $3$-dimensional simplex with vertices
$(0,0,0),(1,0,0)$, $(1,1,0)$ and $(1,1,1)$. Its Ehrhart polynomial is
$E=\binom{x+3}{3}$. The first few values of $\Psi(E^k)$ are
\begin{equation*}
1,
 1/4,
 1/140,
 0,
 -41/60060,
 0,
 50497/19399380,
 0,
 -13687983/148728580,
 0,
 485057494433/30855460020, \dots
\end{equation*}

In the case of even dimension, it seems also that the signs are alternating, see for example \eqref{rama}.

These alternating conjectures are related to the behaviour of the
number of zeroes in $[0,1]$ of the $q$-analogues of these numbers and
to the next topic, namely continuous interpolation of the
Bernoulli-like number by zeta-like functions.

\subsection{Zeta functions of polynomials}

Given a polynomial $E$ in $\QQ[x]$ taking positive values on $\NN$,
one can consider a kind of zeta function attached to $E$, defined by
\begin{equation}
  \label{my_def_zeta}
  Z(E ; s) = \sum_{n \geq 0} \frac{\partial_x E(n)}{E(n)^s},
\end{equation}
for complex numbers $s$ with $\mathfrak{R}(s)>1$.

One will be mostly interested in the case where $E$ is the Ehrhart
polynomial of a lattice polytope $P$. For example, one gets in this way
\begin{equation}
  \sum_{n \geq 0} \frac{1}{(1+n)^s} = \zeta(s) \quad \text{and}\quad
  \sum_{n \geq 0} \frac{2}{(1+2 n)^s} = 2 (1-2^{-s}) \zeta(s)  
\end{equation}
for the two Gorenstein polytopes of dimension $1$.

One shows in the next section that under some mild hypotheses on $E$
the function $Z(E; s)$ is a meromorphic function of the complex
parameter $s$ with only a single pole at $1$ with residue
$1$. Moreover its values at negative integers are given in terms of
$\Psi$ and $E$ by the formula
\begin{equation}
  \label{negative_value}
  Z(E; 1-k) = -\frac{\Psi(E^k)}{k}
\end{equation}
for all $k \in \NN^*$.

Before proving this in the next section, let us give an heuristic
argument. By letting $q=1$ in lemma \ref{psi_and_shift}, one
gets
\begin{equation*}
   \Psi(E(1+x)) - \Psi(E) =  \partial_x E(0).
\end{equation*}
After a telescoping summation, one gets
\begin{equation*}
  \Psi(E(\ell+x)) - \Psi(E) =  \sum_{j=0}^{\ell-1} \partial_x E(j)
\end{equation*}
For polynomials that are powers (of the shape $F^k$ for some $F$), one
therefore gets
\begin{equation}
   \Psi(F(\ell+x)^k) - \Psi(F^k) =  k \sum_{j=0}^{\ell-1} \partial_x F(j) F(j)^{k-1}.
\end{equation}
Formally going to the limit $\ell=\infty$ (and assuming that the first
term of the left-hand side disappears) gives formula \eqref{negative_value}.

\section{Study of zeta-like functions}

\label{meromorph}

Let $(B_k)_{k\geq 0}$ be the sequence of Bernoulli numbers. Recall the linear operator $\Psi : \CC[X] \rightarrow \CC$ defined by
\begin{equation*}
 \Psi (X^k) = B_k \quad \forall k \in \NN.
\end{equation*}
By the usual properties of Bernoulli numbers, there also holds
\begin{equation}
  \label{psi_shifted}
  \Psi ((X+1)^k) = (-1)^k~B_k \quad \forall k \in \NN.
\end{equation}

For all $\x=(x_1,\dots, x_d) \in \RR^d$ and all $\alphab =(\alpha_1,\dots, \alpha_d) \in \NN^d$, we will use in the sequel the following notations:
$ \x^{\alphab} =x_1^{\alpha_1}\dots x_d^{\alpha_d}$ and $ |\alphab|=\alpha_1+\dots +\alpha_d$. One denotes also for any $z\in \CC$ verifying $\Re z >0$ and any $s\in \CC$, $z^s=e^{s \log x}$ where $\log$ is the principal determination of the logarithm.\par

The purpose of this section is to prove the following result:
\begin{theorem}\label{prolong}
 Let $E \in \RR[X]$ be a polynomial of degree $d\geq 1$. 
 Let $a_1,\dots, a_d\in \CC$ be the roots (not necessarily distinct) of $E$.
 Let $A\in \NN^*$ such $\forall x\geq A \enskip \Re E(x) > 0$ .
 One considers the Dirichlet series 
 \begin{equation*}Z_A(E; s):= \sum_{n=A}^{+\infty} \frac{E'(n)}{E(n)^s}.
\end{equation*}
 Then:
 \begin{enumerate} 
 \item $s\mapsto Z_A(E; s)$ converges absolutely in the half-plane $\{ \Re(s) >1\}$ and has a meromorphic continuation to the whole complex plane $\CC$;
 \item the meromorphic continuation of $Z_A(E; s)$ has only one simple pole in $s=1$ with residue $1$.
 \item for any $M\in \NN^*$, $ \quad Z_A(E; 1-M)= -\frac{1}{M} \Psi (E(X+1)^M) -\sum_{n=1}^{A-1} E(n)^{M-1} E'(n)$.
\end{enumerate}
\end{theorem}

\begin{remark}
  By taking $A=1$ and using the shifted polynomial $E(X-1)$ in the
  previous theorem, point 3 gives the formula
  \eqref{negative_value}. Note that summation in $Z_1(E(X-1);s)$ starts at
  $1$, whereas summation in \eqref{my_def_zeta} starts at $0$.
\end{remark}

\begin{remark}
Point 1 of the theorem \ref{prolong} is classic, even in a very general framework (see for example \cite{mahler} or \cite{essouabri}). Our method in this paper is simple and provides, in addition to point 1, the new points 2 and 3 above. In \cite{decrisenoy} an analogue of point 3  was obtained for twisted Dirichlet series. However, the method of \cite{decrisenoy} uses the 
{\bf holomorphy} of twisted Dirichlet series in the whole space and therefore can not be used in our setting here.
\end{remark}

One needs the following elementary lemma:
\begin{lemma}\label{taylorplus}
 Let $d\in \NN^*$ and $\a=(a_1,\dots, a_d)\in \NN^d\setminus \{(0,\dots, 0)\}$. Set $\delta =\left(2 \max_j |a_j|\right)^{-1} >0$. Then, for any $N\in \NN$, any $\s=(s_1,\dots, s_d) \in \CC^d$ and any $x \in [-\delta, \delta]$, we have 
 \begin{equation}\label{taylorrelation}
   \prod_{j=1}^d (1 - x a_j)^{-s_j} = \sum_{\ell=0}^N c_{\ell} (\s)~x^{\ell} + x^{N+1} \rho_N (x; \s)
 \end{equation}
 where
 \begin{equation*}
   c_{\ell}(\s) = (-1)^{\ell} \sum_{\alphab \in \NN^d \atop |\alphab| =\ell } \a^{\alphab} \prod_{j=1}^d \binom{-s_j}{\alpha_j} 
 \end{equation*}
  and
 \begin{equation*}
 \rho_N (x; \s) = (-1)^{N+1} (N+1) \sum_{\alphab \in \NN^d \atop |\alphab| = N+1} \a^{\alphab} \prod_{j=1}^d \binom{-s_j}{\alpha_j} 
 \int_0^1 (1-t)^N \prod_{j=1}^d (1- t x a_j)^{-s_j-\alpha_j} ~dt.
 \end{equation*}
 Moreover we have:
 \begin{enumerate} 
 \item for any $x \in [-\delta, \delta]$, $\s\mapsto \rho_N (\s; x) $ is holomorphic in the whole space $\CC^d$;
 \item for any compact subset $K$ of $\CC^d$, there exists a constant $C=C(K,\a,N,d) >0$ such that 
 \begin{equation*}
   \forall (\s,x) \in K\times [-\delta,\delta] \qquad |\rho_{N} (\s; x)|\leq C. \end{equation*}
 \end{enumerate}
\end{lemma}
\begin{proof}
  Let us fix $\s \in \CC^d$.  One considers the function $\phi$ defined in $
  [-\delta, \delta]$ by $\phi(x) = \prod_{j=1}^d (1 - x
  a_j)^{-s_j}$. The function $\phi$ is infinitely differentiable in $
  [-\delta, \delta]$ and an induction on $\ell$ shows that for all
  $\ell\in \NN$ and all $x \in [-\delta, \delta]$,
 \begin{equation*}
   \frac{\phi^{(\ell)}(x)}{\ell!} =  (-1)^{\ell} \sum_{\alphab \in \NN^d \atop|\alphab| =\ell} \a^{\alphab} \prod_{j=1}^d \binom{-s_j}{\alpha_j} \prod_{j=1}^d (1- x a_j)^{-s_j-\alpha_j}.
 \end{equation*}
 The identity \eqref{taylorrelation} then follows from the application
 of the Taylor formula with integral remainder at $x=0$.

 The second part of the lemma follows from the theorem of holomorphy
 under the integral sign. This completes the proof of Lemma
 \ref{taylorplus}.
\end{proof}

\begin{proof}
  {\bf (points 1 and 2 of Theorem \ref{prolong})}\\
  For short, let us write $Z(s)$ for $Z_A(E;s)$.

  First we remark that if $\a=(a_1,\dots, a_d)=(0,\dots, 0)$, then $E$
  is of the form $E(X) = u X^d$ where $u>0$.
  It follows that $Z(s)=d u^{1-s} \zeta (d s-d+1)$ and Theorem \ref{prolong} is true in this case.

  One will assume in the sequel that $\a \neq (0,\dots,0)$ and set $\delta =\left(2 \max_j |a_j|\right)^{-1} >0$.
  One will note in the sequel $s=\sigma+ i\tau$ where $\sigma =\Re (s)$ and $\tau =\Im (s)$.
  It is easy to see that
\begin{equation*}
  \left|\frac{E'(n)}{E(n)^s}\right| \ll \frac{1}{n^{d\sigma -(d-1)}}.
\end{equation*}
It follows that $s\mapsto Z(s)$ converges absolutely in the half-plane $\{ \Re(s) >1\}$.\\

As the act of removing or adding a finite number of terms does not
change the meromorphy or poles, we can choose the integer $ A $ as
large as possible.
Let us choose here $A \in \NN^*$ such that $A \geq 2 \sup_{1\leq j \leq d} |a_j| = \delta^{-1}$.\\
It is clear that we can also assume without loss of generality that
the polynomial $E$ is unitary. It follows that
\begin{equation*}E(X)= \prod_{j=1}^d (x-a_j) \quad {\mbox { and }} \quad E'(X) = E(X) \left( \sum_{j=1}^d \frac{1}{X-a_j}\right).\end{equation*}
One deduces that for all $s \in \CC$ satisfying $\sigma = \Re (s)> 1$ there holds:
\begin{eqnarray*}
Z(s)&=& \sum_{n=A}^{+\infty} \frac{E'(n)}{E(n)^s}
= \sum_{j=1}^d \sum_{n=A}^{+\infty} \frac{1}{(n-a_j)^s \prod_{k\neq j} (n - a_k)^{s-1}}\\
&=& \sum_{j=1}^d \sum_{n=A}^{+\infty} 
\frac{1}{n^{ds -(d-1)}} \left(1-\frac{a_j}{n} \right)^{-s} \prod_{k\neq j} \left(1-\frac{a_k}{n} \right)^{-s+1}.
\end{eqnarray*}
Let $N\in \NN$. Lemma \ref{taylorplus} and the previous relation imply that for all $s\in \CC$ verifying $\sigma =\Re(s) >1$ we have:
\begin{equation*}
Z(s)= \sum_{j=1}^d \sum_{n=A}^{+\infty} 
\frac{1}{n^{ds -(d-1) }} \left[\sum_{\ell=0}^N c_{\ell}\left(f_j(s)\right) \frac{1}{n^{\ell}} + \frac{1}{n^{N+1}} \rho_N \left(x, f_j(s)\right)\right],
\end{equation*} 
where $f_j(s) =(s_1,\dots,s_d)$ with $s_k =s-1$ if $k\neq j$ and $s_j=s$.\\

One deduces that for all $s\in \CC$ verifying $\sigma =\Re(s) >1$ there holds:
\begin{eqnarray}\label{bon1}
Z(s)&=& \sum_{\ell=0}^N \left[\sum_{j=1}^dc_{\ell}\left(f_j(s)\right)\right] \zeta_A\left(d s -(d-1)+\ell\right) \nonumber \\
& & + \sum_{n=A}^{+\infty} 
\frac{1}{n^{ds -(d-1)+N+1}} \left[\sum_{j=1}^d \rho_N \left(x; f_j(s)\right)\right],
\end{eqnarray}
where $\zeta_A(s):= \sum_{n=A}^{+\infty} \frac{1}{n^s} =\zeta (s) - \sum_{n=1}^{A-1} \frac{1}{n^s}$.\\

On the other hand it is easy to see that for all $\ell\in \NN$:
\begin{eqnarray}\label{bon2}
\sum_{j=1}^d c_{\ell}\left(f_j(s)\right) &=& \sum_{j=1}^d (-1)^{\ell} \sum_{\alphab \in \NN^d \atop |\alphab| =\ell } \a^{\alphab} \binom{-s}{\alpha_j} \prod_{k\neq j} \binom{-s+1}{\alpha_k} \nonumber \\
&=& \sum_{j=1}^d (-1)^{\ell} \sum_{\alphab \in \NN^d \atop |\alphab| =\ell } \a^{\alphab} \frac{s+\alpha_j-1}{s-1} \prod_{k=1}^d \binom{-s+1}{\alpha_k} \nonumber \\
&=& (-1)^{\ell} \sum_{\alphab \in \NN^d \atop |\alphab| =\ell } \a^{\alphab} \prod_{k=1}^d \binom{-s+1}{\alpha_k} \sum_{j=1}^d \frac{s+\alpha_j-1}{s-1} \nonumber \\
&=& (-1)^{\ell} \sum_{\alphab \in \NN^d \atop |\alphab| =\ell } \a^{\alphab} \prod_{k=1}^d \binom{-s+1}{\alpha_k} \frac{d s-d + \ell}{s-1}.
\end{eqnarray}
Relations \eqref{bon1} and \eqref{bon2} imply that for all $s\in \CC$ satisfying $\sigma =\Re(s) >1$ we have:
\begin{eqnarray}\label{bon3}
(s-1) Z(s)&=& \sum_{\ell=0}^N \left[ (-1)^{\ell} \sum_{\alphab \in \NN^d \atop |\alphab| =\ell } \a^{\alphab} \prod_{k=1}^d \binom{-s+1}{\alpha_k}\right] \left(d s-d + \ell\right) 
\zeta_A\left(d s -(d-1)+\ell\right) \nonumber \\
& & + (s-1) \sum_{n=A}^{+\infty} 
\frac{1}{n^{ds -(d-1)+N+1}} \left[\sum_{j=1}^d \rho_N \left(x; f_j(s)\right)\right].
\end{eqnarray}
Moreover,
\begin{enumerate}
\item the point 2 of lemma \ref{taylorplus} and the the dominated convergence theorem of Lebesgue imply that 
\begin{equation*}s\mapsto \sum_{n=A}^{+\infty} 
\frac{1}{n^{ds -(d-1)+N+1}} \left[\sum_{j=1}^d \rho_N \left(x; f_j(s)\right)\right]\end{equation*}
is defined and is holomorphic in the half-plane $\{\sigma > 1-\frac{N+1}{d}\}$;
\item the classical properties of the Riemann zeta function imply that the function $s\mapsto (s-1)\zeta_A(s)$ is holomorphic in the whole complex plane $\CC$.
\end{enumerate}
These last two points and identity \eqref{bon3} implies that $s \mapsto (s-1) Z (s) $ has a holomorphic extension to the half-plane $\{\sigma > 1-\frac{N+1}{d}\}$.
As $N\in \NN$ is arbitrary, we deduce that $s\mapsto (s-1) Z(s)$ has a holomorphic continuation to the whole complex plane $\CC$.\\
It follows that $s\mapsto Z(s)$ has a meromorphic continuation to the whole complex plane $\CC$ with at most one possible simple pole in $s=1$.\\
So to finish the proof of points 1 and 2 of Theorem \ref{prolong}, it suffices to show that $s = 1$ is a pole of residue $1$.
But relation \eqref{bon3} with $N=0$ implies that 
\begin{equation*}\lim_{s\to 1} (s-1) Z(s) = \lim_{s\to 1} (d s -d) \zeta_A (d s -d +1) =1.\end{equation*}
One deduces that $s=1$ is a simple pole of $Z(s)$ and that $Res_{s=1} Z(s) =1$. This completes the proof of points 1 and 2 of Theorem \ref{prolong}.
\end{proof}

\begin{proof}
  {\bf (point 3 of Theorem \ref{prolong})}\\
  First let us recall the classical formula 
  \begin{equation}\label{zetanegative}
    k ~\zeta (1-k) =(-1)^{k-1} B_k = - \Psi((X+1)^k) \quad \forall k\in \NN^*.
  \end{equation}
Formula (\ref{zetanegative}) is also valid for $k=0$ by analytic continuation. We will use this fact in the sequel.\\

\noindent
Let $M \in \NN^*$. Set $N= d M$. In particular, $1 - M >
  1-\frac{N+1}{d}$. If $|\alphab|=N+1$, then for any $j=1,\dots, d$:
  \begin{equation*}
    \alpha_j >M-1 \quad {\mbox { or }} \quad {\mbox { there exists }} k \in \{1,\dots d\}\setminus \{j\} {\mbox { such that }} \alpha_k >M.
  \end{equation*}
  One deduces that for any $j=1,\dots, d$: 
  \begin{eqnarray*}
    \rho_N \left(x; f_j(1-M)\right) &=& (-1)^{N+1} (N+1) \sum_{\alphab \in \NN^d \atop |\alphab| = N+1} \a^{\alphab}\binom{M-1}{\alpha_j} \prod_{k\neq j}^d \binom{M}{\alpha_k} \\
    & & \times 
    \int_0^1 (1-t)^N (1- t x a_j)^{M-\alpha_j} \prod_{k\neq j}^d (1- t x a_k)^{M-1-\alpha_k} ~dt\\
    &=& 0.
  \end{eqnarray*}
  It follows then from \eqref{bon3} that 
  \begin{eqnarray}\label{neg1}
    -M Z(1-M)&=& \sum_{\ell=0}^N \left[ (-1)^{\ell} \sum_{\alphab \in \NN^d \atop |\alphab| =\ell } \a^{\alphab} \prod_{k=1}^d \binom{M}{\alpha_k}\right] \left(\ell-d M\right) 
    \zeta_A\left(1+\ell-d M\right) \nonumber \\
    &=& \sum_{\ell=0}^N \left[ (-1)^{\ell} \sum_{\alphab \in \NN^d \atop |\alphab| =\ell } \a^{\alphab} \prod_{k=1}^d \binom{M}{\alpha_k}\right] \left(\ell-d M\right) 
    \zeta \left(1+\ell-d M\right) \nonumber\\
    & & -\sum_{\ell=0}^N \left[ (-1)^{\ell} \sum_{\alphab \in \NN^d \atop |\alphab| =\ell } \a^{\alphab} \prod_{k=1}^d \binom{M}{\alpha_k}\right] \left(\ell-d M\right) 
    \left(\sum_{u=1}^{A-1} u^{d M -\ell -1}\right).
  \end{eqnarray}
  This sum therefore splits into two parts. Remarking that if
  $|\alphab| > N$ then there exists $k$ such that $\alpha_k > M$ and
  hence $ \binom{M}{\alpha_k}=0$, one can compute the second part:
  \begin{eqnarray}\label{neg2}
    \kappa&:=& -\sum_{\ell=0}^N \left[ (-1)^{\ell} \sum_{\alphab \in \NN^d \atop |\alphab| =\ell } \a^{\alphab} \prod_{k=1}^d \binom{M}{\alpha_k}\right] \left(\ell-d M\right) 
    \left(\sum_{u=1}^{A-1} u^{d M -\ell -1}\right)\\
    &=&\sum_{u=1}^{A-1} u^{d M -1} \sum_{\ell=0}^N \sum_{\alphab \in \NN^d \atop |\alphab| =\ell } (d M -\sum_{j=1}^d \alpha_j) 
    \left[ \prod_{k=1}^d \binom{M}{\alpha_k}\left(-\frac{a_k}{u}\right)^{\alpha_k}\right]\nonumber\\
&=& \sum_{u=1}^{A-1} u^{d M -1} \sum_{\alphab \in \{0,\dots, M\}^d} \ (d M -\sum_{j=1}^d \alpha_j) 
\left[ \prod_{k=1}^d \binom{M}{\alpha_k}\left(-\frac{a_k}{u}\right)^{\alpha_k}\right].\nonumber
\end{eqnarray}
Continuing this computation by splitting this sum in two, we have
\begin{eqnarray}\label{neg3}
\kappa&=& d M \sum_{u=1}^{A-1} u^{d M -1} 
\prod_{k=1}^d \left(1-\frac{a_k}{u}\right)^{M}
 -\sum_{u=1}^{A-1} \sum_{j=1}^d u^{d M -1} \frac{M\left(-\frac{a_j}{u}\right)}{1-\frac{a_j}{u}}\prod_{k=1}^d \left(1-\frac{a_k}{u}\right)^{M}\nonumber\\
&=& d M \sum_{u=1}^{A-1} u^{-1} E(u)^M+ M \sum_{u=1}^{A-1} \sum_{j=1}^d \frac{a_j}{u(u-a_j)} E(u)^M \nonumber\\
&=& M \sum_{u=1}^{A-1} E(u)^{M-1} E'(u).
\end{eqnarray}

Relations \eqref{zetanegative}, \eqref{neg1}, \eqref{neg2} and \eqref{neg3} imply that 
\begin{equation}\label{fin1}
Z(1-M)= \frac{(-1)^{d M -1}}{M} \sum_{\ell=0}^N \left[ \sum_{\alphab \in \NN^d \atop |\alphab| =\ell } \a^{\alphab} \prod_{k=1}^d \binom{M}{\alpha_k}\right] B_{d M -\ell} 
-\sum_{u=1}^{A-1} E(u)^{M-1} E'(u).
\end{equation}
On the other hand, it is easy to see that 
\begin{eqnarray*}
E(X)^M&=& \prod_{j=1}^d (X-a_j)^M= \prod_{j=1}^d \left( \sum_{\alpha_j =0}^M \binom{M}{\alpha_j} (-a_j)^{\alpha_j} X^{M-\alpha_j}\right)\\
&=& \sum_{\alphab \in \{0,\dots, M\}^d} (-1)^{|\alphab|} \a^{\alphab} \left(\prod_{j=1}^d \binom{M}{\alpha_j}\right) X^{d M -|\alphab|}
= \sum_{\ell=0}^N (-1)^{\ell} \left[\sum_{\alphab \in \NN^d \atop |\alphab|=\ell} \a^{\alphab} \prod_{j=1}^d \binom{M}{\alpha_j}\right] X^{d M -\ell}.
\end{eqnarray*}
Using \eqref{psi_shifted}, it follows that 
\begin{equation*}\Psi(E(X+1)^M)=(-1)^{d M} \sum_{\ell=0}^N \left[\sum_{\alphab \in \NN^d \atop |\alphab|=\ell} \a^{\alphab} \prod_{j=1}^d \binom{M}{\alpha_j}\right] B_{d M -\ell}.\end{equation*}
One then deduces from \eqref{fin1} that 
$Z(1-M)= -\frac{1}{M} \Psi(E(X+1)^M) -\sum_{u=1}^{A-1} E(u)^{M-1} E'(u)$. This completes the proof of Theorem \ref{prolong}.
\end{proof}

\bibliographystyle{plain}
\bibliography{goren}

\end{document}